\documentclass[12pt,reqno]{amsart}
\usepackage{amssymb,color}

\usepackage[dvips]{epsfig}

\newcommand{\e}{\varepsilon}

\newcommand{\M}{\mathcal{M}}
\newcommand{\la}{\lambda}
\newcommand{\al}{\alpha}

\newcommand{\fy}{\varphi}

\newcommand{\p}{\partial}
\newcommand{\I}{\infty}

\newcommand{\R}{\mathbb{R}}

\newcommand{\C}{\mathbb{C}}

\renewcommand{\Im}{\mathop{\mathrm{Im}}}
\renewcommand{\bar}{\overline}
\renewcommand{\hat}{\widehat}

\renewcommand{\r}{\rho}

\newcommand{\Sg}{\mathfrak{S}}

\numberwithin{equation}{section}

\newtheorem{thm}{Theorem}[section]
\newtheorem{cor}[thm]{Corollary}
\newtheorem{lem}[thm]{Lemma}
\newtheorem{prop}[thm]{Proposition}
\theoremstyle{remark}
 
\newcommand{\ran}{\rangle}
\newcommand{\lan}{\langle}

\newcommand{\lec}{\lesssim}
\newcommand{\gec}{\gtrsim}

\newcommand{\EQ}[1]{\begin{equation} \begin{split} #1 \end{split} \end{equation}}
\newcommand{\Del}[1]{}
\newcommand{\CAS}[1]{\begin{cases} #1 \end{cases}}
\newcommand{\pt}{&}
\newcommand{\pr}{\\ &}
\newcommand{\pq}{\quad}

\newcommand{\LR}[1]{{\lan #1 \ran}}
\newcommand{\de}{\delta}
\newcommand{\si}{\sigma}

\renewcommand{\t}{\tau}

\newcommand{\ka}{\kappa}
\newcommand{\ga}{\gamma}
\newcommand{\x}{\xi}

\newcommand{\na}{\nabla}

\newcommand{\Cu}{\bigcup}

\newcommand{\sg}{\mathfrak{s}}
\newcommand{\HH}{\mathcal{H}}

\def\dist{\mathrm{dist}}
\def\calH{{\mathcal H}}

\def\nn{\nonumber}

\def\eps{\varepsilon}
\def\sign{\mathrm{sign}}

\def\calM{{\mathcal M}}

\def\calL{{\mathcal L}}
\def\diag{\mathrm{diag}}
\def\calF{\mathcal{F}}

\author{J.~Krieger}
\address{ B\^atiment des Math\' ematiques,
Station 8, 
CH-1015 Lausanne, Suisse} 
\email{ joachim.krieger@epfl.ch  }

\author{K.~Nakanishi}
\address{Department of Mathematics, Kyoto University\\ Kyoto 606-8502, Japan}
\email{n-kenji@math.kyoto-u.ac.jp}

\author{W.~Schlag}
\address{Department of  Mathematics, The University of Chicago\\ Chicago, IL 60615, U.S.A.} 
\email{schlag@math.uchicago.edu} 

\thanks{The third author was supported by the NSF, DMS-0617854, and a Guggenheim fellowship.}

\title[Global dynamics above the ground state energy]{Global dynamics above the ground state energy\\ for the one-dimensional  NLKG equation}

\begin{document}




\begin{abstract}
In this paper we obtain a global characterization of the dynamics of even solutions to the one-dimensional
nonlinear Klein-Gordon (NLKG) equation on the line with focusing nonlinearity $|u|^{p-1}u$, $p>5$, provided their energy exceeds that
of the ground state only sightly. The method is the
same as in the three-dimensional case~\cite{NakS}, the major difference being in the construction of the center-stable manifold.
The difficulty there lies with the weak dispersive decay of $1$-dimensional NLKG. In order to address this specific issue, we
establish local dispersive estimates for the perturbed linear Klein-Gordon equation, similar to those of Mizumachi~\cite{Miz}.
The essential ingredient for the latter class of estimates is the absence of a threshold resonance of the linearized operator. 
\end{abstract}

\maketitle

\section{Introduction}

In this paper we address the long-time evolution of solutions to the nonlinear Klein-Gordon equation
\EQ{
\label{eq:NLKG}
 u_{tt} - u_{xx} + u = |u|^{p-1} u
}
in $1+1$ dimensions, and with $p>5$. It is standard that this equation is locally wellposed in the energy class $H^{1}\times L^{2}$
and can exhibit both global existence (for example for small data) as well as finite time blowup (for example, for negative energy). 
In addition, equation~\eqref{eq:NLKG} admits explicit soliton solutions
\EQ{
\label{eq:soliton}
Q(x) =  \alpha \cosh^{-\frac{1}{\beta}}(\beta x),\qquad \al=\big(\frac{p+1}{2}\big)^{\frac{1}{p-1}}, \;\; \beta=\frac{p-1}{2}
}
Our goal in this paper is to prove the following result, where 
\EQ{
E(u,\dot u) = \int_{-\I}^{\I} \big[ \frac12(|\p_{x} u|^{2} + |\dot u|^{2} + |u|^{2}) - \frac{1}{p+1} |u|^{p+1} \big]\, dx
}
is the conserved energy. 

\begin{thm}
\label{thm:main} 
Let $p>5$. There exists $\eps>0$ such that any even data $(u_0,u_1)\in H^1\times L^2(\R)$ with energy
\EQ{
\label{eq:Econd}
E(u,\dot u)< E(Q,0) + \eps^2
}
have the property that the solutions $u(t)$ of~\eqref{eq:NLKG} associated with these data  exhibit the following trichotomy:
\begin{itemize}
\item $u$ blows up in finite positive time
 \item $u$ exists globally and scatters to zero as $t\to\I$
\item $u$ exists globally and scatters to $Q$, i.e., there exist $(v_0,v_1)\in H^1\times L^2$ with the property that
\EQ{
\label{eq:Qscatters}
(u(t),\dot u(t)) = (Q,0) + (v(t),\dot v(t)) + o_{\calH}(1)\qquad t\to\I
}
for some free Klein-Gordon solution $(v(t),\dot v(t))\in H^1\times L^2$. 
\end{itemize}
In addition, the set of even data as in~\eqref{eq:Econd} splits into nine nonempty disjoint sets corresponding to all possible combinations
of this trichotomy as $t\to\pm\I$. 
\end{thm}

Moreover, we obtain a characterization of the threshold solutions, i.e., those with energies $E(\vec u)=E(Q,0)$, cf.~\cite{DM1}, \cite{DM2}.
In fact, we find the following.

\begin{cor}
\label{cor:DM} 
The even solutions to \eqref{eq:NLKG} with energy   $E(\vec u)=E(Q,0)$ can be characterized as follows: 
\begin{itemize}
\item
 they blow up in both
the positive and negative time directions
\item the exist globally on $\R$ and scatter as $t\to\pm\I$
\item they are constant $\pm Q$
\item they equal one of the following solutions, for some $t_{0}\in\R$: 
\[
\si W_{+}(t+t_{0}, x),\quad \si W_{-}(t+t_{0},x) , \quad \si W_{+}(-t+t_{0}, x),\quad \si W_{-}(-t+t_{0},x) 
\]
where $(W_{\pm}(t,\cdot),\p_{t } W_{\pm}(t,\cdot))$ approach $(Q,0)$ exponentially fast in $\HH$ as $t\to\I$, and $\si=\pm1$. 
As $t\to-\I$, $W_{+}$ scatters to zero, whereas $W_{-}$ blows up in finite time. 
\end{itemize}
\end{cor} 

As usual, the image of $W_{\pm}$ and $Q$ form a one-dimensional {\em stable manifold} associated with $(Q,0)$, cf.~\cite{BJ}. The {\em unstable manifold}
is obtained by time-reversal. 
Moreover, the solutions in Theorem~\ref{thm:main} which scatter to $Q$ form a  $C^1$   manifold\footnote{It is in fact smoother than $C^1$, but we do not pursue this here.} in $\HH$ of co-dimension $1$ which is
the {\em center-stable manifold} associated with $(Q,0)$. The {\em center manifold} is obtained by the transverse intersection of the two center-stable
manifolds corresponding to $t\to\pm\I$, respectively.
We remark that the case of energies $E(\vec u)<E(Q,0)$ was settled in~\cite{PS}. There it was found that one either has global existence
or blowup in both time directions. That global existence entails scattering to zero was only recently shown in~\cite{IMN}. 

The restriction $p>5$ in Theorem~\ref{thm:main} appears to be a technical one; in fact, it seems reasonable to believe that Theorem~\ref{thm:main} remains
valid in the range $p>3$, and possibly even for some $p\le3$. At $p=3$ the linearized operator $L_+$ has a threshold resonance which might affect the result. But more
importantly, in the range $p<5$ small data scattering cannot be done on the basis of Strichartz estimates alone which is of course a serious obstacle at this point. 
The case $p=5$ is perhaps accessible, but we exclude it here (as in~\cite{IMN}). 
We remark that unlike the NLS case, the hyperbolic structure underlying our proof of Theorem~\ref{thm:main} is still present for all $1<p\le 5$. 
This refers to the fact that the structure of the spectrum of the linearized form of \eqref{eq:NLKG} around $Q$ does not change significantly.  


This paper is closely related to~\cite{NakS}, and for many details in the following section we refer the reader to that paper. The main difference
lies with the construction of the center-stable manifold. The radial assumption in~\cite{NakS} was removed in~\cite{NakS3}, and we expect similarly
that one can obtain Theorem~\ref{thm:main} for all data, and not just even ones; but we do not pursue that here. 
Finally, see~\cite{NakS2} for similar results on three-dimensional cubic NLS, and \cite{KNS} for the critical wave equations in three and five dimensions.

\section{Global existence vs. finite time blowup}

In this section we recall the statement and proof of the ``nine set theorem''
from~\cite{NakS}. That theorem is weaker than Theorem~\ref{thm:main} in the sense
that ``scattering to $Q$'' is replaced by ``trapped by $Q$'', which means that
the solution remains in a small neighborhood of $Q$ for all times (positive or negative).
In the following section we shall than take up the issue of proving that trapped by~$Q$
actually implies the stronger scattering  property. This amounts to proving the existence
of a center-stable manifold as in~\cite{NakS}. It is at that point that this paper
differs most from~\cite{NakS}, due to the fact that the one-dimensional equation exhibits much less
dispersion than the three-dimensional one.  Let $\vec u=(u,\dot u)$ and set 
\EQ{ \label{def He}
 \HH^\e:=\{\vec u\in \HH=H^{1}\times L^{2} \mid E(\vec u)<E(Q,0)+\e^2, \vec u(t,x)=\vec u(t,-x)\}
 }

\begin{thm}
\label{thm:9set} 
Consider all solutions of NLKG~\eqref{eq:NLKG} with initial data $\vec u(0)\in\HH^\e$ for some small $\e>0$. The solution set is decomposed into nine non-empty sets characterized as 
\begin{enumerate}
\item Scattering to $0$ for both  $t\to\pm\I$, 
\item Finite time  blow-up on both sides $\pm t>0$, 
\item Scattering to $0$ as $t\to\I$ and finite time  blow-up in $t<0$, 
\item Finite time  blow-up in $t>0$ and scattering to $0$ as $t\to-\I$, 
\item Trapped by $\pm Q$ for $t\to\I$ and scattering to $0$ as $t\to-\I$, 
\item Scattering to $0$ as $t\to\I$ and trapped by $\pm Q$ as $t\to-\I$, 
\item Trapped by $\pm Q$ for $t\to\I$ and finite time  blow-up in $t<0$, 
\item Finite time  blow-up in $t>0$ and trapped by $\pm Q$ as $t\to-\I$, 
\item Trapped by $\pm Q$  as $t\to\pm\I$, 
\end{enumerate}
where ``trapped by $\pm Q$" means that the normalized  solution stays in a $O(\e)$ neighborhood of $ \pm Q $ forever after some time (or before some time). 
The initial data sets for (1)-(4), respectively, are open. 
\end{thm}

\subsection{Hyperbolic and variational structures}

We now recall the main steps in the three-dimensional proof from~\cite{NakS}, which in essence carries over to
the one-dimensional case with only minimal changes. Next to the conserved energy $E$, we require the ``static energy'' or ``action''
\EQ{ \label{eq:stat energy}
 J(u) = \int \big[ \frac12(|\p_{x} u|^{2}   + |u|^{2}) - \frac{1}{p+1} |u|^{p+1} \big]\, dx
}
and the functionals
\EQ{
K_{0}(u) &= \lan J'(u)| u\ran = \int  (|\p_{x} u|^{2}  +|u|^{2}  -  |u|^{p+1})\, dx \\
K_{2}(u) &= \lan J'(u)| A\ran = \int  (|\p_{x} u|^{2}    - \frac{p-1}{2(p+1)} |u|^{p+1})\, dx  
}
with $A=\frac12(x\p_{x} + \p_{x}x)$ the generator of dilations.   
It is standard that $\pm Q(\cdot + x_{0})$ as in~\eqref{eq:soliton} are the only solutions of the static equation
\[
-\phi'' + \phi = |\phi|^{p-1}\phi
\]
which are in $H^{1}$, and they uniquely minimize $J(\phi)$ subject to $K_{s}(\phi)=0$, $s=0,2$. 
Setting
\EQ{ \label{decop u}
 u = \si[Q+v], \pq v=\la\r+\ga, \pq \ga\perp\r}
 where $\si=\pm$, one obtains 
 \EQ{
 E(\vec u) &= J(Q) +\frac12(\dot\la^{2}-k^{2}\la^{2})+ \frac12 (\lan L_{+} \ga|\ga\ran + \| \dot\ga\|_{2}^{2}) - C(v) 
 }
 where $C(v)=O(\|v\|_{H^{1}}^{3})$. We now define
 \EQ{
 \| v\|_{E}^{2} &:= \frac12(\dot\la^{2}+k^{2}\la^{2})+ \frac12 (\lan L_{+} \ga|\ga\ran + \| \dot\ga\|_{2}^{2})  \\
 d_{\si}(\vec u) &:= \sqrt{  \| \vec v\|_{E}^{2} - \chi(\|\vec v\|_{E}/(2\de_{E})) C(V) }
 }
where $0<\de_{E}\ll 1$ and $\chi$ is a suitable cut-off function so that 
\EQ{
  \|\vec v\|_E/2 \le d_\si(\vec u) \le 2\|\vec v\|_E,\pq d_\si(\vec u)=\|\vec v\|_E+O(\|\vec v\|_E^2),}
\EQ{ \label{energy dist}
 d_\si(\vec u)\le \de_E \implies d_\si(\vec u)^2=E(\vec u)-J(Q)+k^2\la^2.}

The {\em eigenmode dominance} and the {\em ejection process} from~\cite{NakS} carry over to this setting without changes.
For the sake of completeness, we reproduce these statements here, but refer the reader to~\cite{NakS} for the proofs. 

\begin{lem} \label{1st exit}
For any $\vec u\in\HH$ satisfying  
\EQ{
 E(\vec u)<J(Q)+d_Q(\vec u)^2/2, \pq d_Q(\vec u)\le \de_E,}
one has $d_Q(\vec u)\simeq|\la|$.  
\end{lem}

The following {\em ejection lemma} is most important in the proof of Theorem~\ref{thm:9set} and Theorem~\ref{thm:main}. It states
that a solution which penetrates a small neighborhood of $\pm(Q,0)$ deep enough (compared to $\e$) in~$\HH^{\e}$, but which is not trapped by
these equilibria, is necessary ejected from the neighborhood via the exponentially increasing mode. 
The nontrapping here is formulated via the condition $d_{Q}(\vec u(t))\ge d_{Q}(\vec u(0))$ for small $t\ge0$, cf.~\eqref{exiting condition}.

\begin{lem} \label{2nd exit}
There exists a constant $0<\de_X\le \de_E$ with the following property. Let $u(t)$ be a local solution of NLKG on an interval $[0,T]$ satisfying 
\EQ{\nn 
 R:=d_Q(\vec u(0)) \le \de_X, \pq E(\vec u)<J(Q)+R^2/2} 
and for some $t_0\in(0,T)$, 
\EQ{ \label{exiting condition}
 d_Q(\vec u(t)) \ge R \pq (0<\forall t<t_0).}
Then $d_Q(\vec u(t))$ increases monotonically until reaching $\de_X$, and meanwhile, 
\EQ{\nn 
  \pt d_Q(\vec u(t)) \simeq -\sg\la(t) \simeq -\sg\la_+(t) \simeq e^{kt}R, 
  \pr |\la_-(t)|+\|\vec \ga(t)\|_E \lec R+e^{2kt}R^2, 
  \pr \min_{s=0,2}\sg K_s(u(t)) \gec d_Q(\vec u(t)) - C_*d_Q(\vec u(0)),}
for either $\sg=+1$ or $\sg=-1$, where $C_*\ge 1$ is a constant. 
\end{lem} 

One point in which the proofs depend on the dimension and the power of the nonlinearity is the linearization of the $K$ functionals.
Here they are
\EQ{
K_{0}(Q+v) &= -(p-1) \lan Q^{p} |v\ran + O(\|v\|_{H^{1}}^{2}) \\
K_{2}(Q+v) &= - \lan \frac{p-5}{2} Q^{p} + 2 Q|v\ran + O(\|v\|_{H^{1}}^{2})
}

Next, we require the following variational lower bounds, which are proved as in dimension three (but one needs to use the concentration compactness principle
as in~\cite{NakS3}).

\begin{lem} \label{K lower bd}
For any $\de>0$, there exist $\e_0(\de), \ka_0, \ka_1(\de)>0$ such that for any $\vec u\in\HH$ satisfying 
\EQ{ \label{energy region}
 E(\vec u)< J(Q)+\e_0(\de)^2, \pq d_Q(\vec u) \ge \de,}
one has either
\EQ{ \label{-K bd}
 K_0(u) \le -\ka_1(\de) \pq and \pq K_2(u) \le -\ka_1(\de),}
or 
\EQ{ \label{+K bd}
 K_0(u) \ge \min(\ka_1(\de), \ka_0\|u\|_{H^1}^2) \pq and \pq K_2(u) \ge \min(\ka_1(\de),\ka_0\| \p_{x}u\|_{L^2}^2).}
\end{lem}

Combining the signs of $\la$ and $K_{s}$ one now obtains a global sign functional.

\begin{lem} \label{sign}
Let $\de_S:=\de_X/(2C_*)>0$ where $\de_X$ and $C_*\ge 1$ are constants from Lemma \ref{2nd exit}. Let $0<\de\le\de_S$ and 
\EQ{ \label{def HS}
 \HH_{(\de)}:=\{\vec u\in\HH \mid E(\vec u)<J(Q)+\min(d_Q(\vec u)^2/2,\e_0(\de)^2)\},} 
where $\e_0(\de)$ is given by Lemma \ref{K lower bd}. Then there exists a unique continuous function $\Sg:\HH_{(\de)}\to\{\pm 1\}$ satisfying 
\EQ{ \label{def Sg}
 \CAS{\vec u\in\HH_{(\de)},\ d_Q(\vec u)\le\de_E &\implies  \Sg(\vec u)=-\sign\la,\\
 \vec u\in\HH_{(\de)},\ d_Q(\vec u)\ge\de &\implies \Sg(\vec u)=\sign K_0(u)=\sign K_2(u),}}
where we set $\sign 0=+1$ (a convention for the case $u=0$). 
\end{lem}

The free energy is uniformly bounded in the region $\Sg=1$. More precisely, one has the following result. 

\begin{lem} \label{Sg+ eng bd}
There exists $M_*\sim J(Q)^{1/2}$ such that for any $\vec u\in\HH_{(\de_S)}$ satisfying $\Sg(\vec u)=+1$ we have $\|\vec u\|_\HH \le M_*$. 
\end{lem}

For the proofs of these statements, see~\cite{NakS}.

\subsection{Absence of almost homoclinic orbits}

In the section, we recall the {\em one-pass} theorem from~\cite{NakS}. We present the proof of this result for the
sake of completeness. We begin with the following result on global scattering which excludes the degenerate case of $K_{2}(u(t))>0$ 
being to small for too long. 

\begin{lem} \label{0freq scat}
For any $M>0$, there exists $\mu_0(M)>0$ with the following property. Let $u(t)$ be a finite energy solution of NLKG \eqref{eq:NLKG} on $[0,2]$ satisfying 
\EQ{ \label{low freq conc} 
 \|\vec u\|_{L^\I_t(0,2;\HH)} \le M, \pq 
 \int_0^2 \|\p_{x} u(t)\|_{L^2}^2 \, dt \le \mu^2}
for some $\mu\in(0,\mu_0]$. Then $u$ extends to a global solution and scatters to $0$ as $t\to\pm\I$, and moreover $\|u(t)\|_{L^p_tL^{2p}_x(\R\times\R^3)}\ll\mu^{\kappa}$
for some constant $\kappa=\kappa(p)>0$. 
\end{lem}
\begin{proof}
First we see that $u$ can be approximated by the free solution 
\EQ{\nn 
 v(t):=e^{i\LR{\na}t}v_+ + e^{-i\LR{\na}t}v_-,
 \pq v_\pm:=[u(0)\mp i\LR{\na}^{-1}\dot u(0)]/2.}
This follows  simply from the Duhamel formula 
\EQ{\nn 
 \|v-u\|_{L^\I_tH^1_x(0,2)} \pt\lec \|\, |u|^{p-1}u \, \|_{L^1_t((0,2),L^2_x)}
 \lec \|u\|_{L^p_tL^{2p}_x}^p 
 \pr\lec \|\p_{x} u\|_{ L^2_t((0,2),L^2_x)}^2 \| u\|_{L^\I_t((0,2),\HH)}^{p-2} \le \mu^2 M^{p-2} \ll \mu,}
if $\mu_0 M^{p-2}\ll 1$, where we used H\"older's inequality, Sobolev embedding, and $p\ge5$. In particular, 
\EQ{\nn 
 4\mu^2 \pt\ge \int_0^2\|\p_{x} v(t)\|_{L^2_x}^2 \, dt
 \pr=C\int|\x|^2\left[2|\hat v_+|^2+2|\hat v_-|^2+\Im\{\LR{\x}^{-1}(e^{4i\LR{\x}}-1)\hat v_+\bar{\hat v_-}\}\right]\, d\x
 \pr\gec\|\p_{x} v_+\|_{L^2}^2+\|\p_{x} v_-\|_{L^2}^2,}
where $\hat v$ denotes the Fourier transform in $x$ of $v$. 
Now we use the Strichartz estimate for the free Klein-Gordon equation
\EQ{\label{eq:KGS}
 \|e^{\pm i\LR{\na}t}\fy\|_{  L^{p}_{t} B^{\al}_{q,2} } \lec \|\fy\|_{H^1_x}}
 with $\al=1-\frac{3}{p}$, $\frac{1}{q}=\frac12-\frac{2}{p}$, 
see~\eqref{eq:STR2} below. 
Combining it with Sobolev, we obtain with $p>5$ fixed, $\al, q$ as in~\eqref{eq:KGS}, and $\beta=\frac12(1-\frac{5}{p})>0$, 
\EQ{\nn 
 \|v\|_{L^p_tL^{2p}_x(\R\times\R)} \pt\lec \|v\|_{L^p_t\dot B^{\beta}_{q,2}(\R\times\R)} \lec \sum_\pm\|v_\pm\|_{\dot H^{\beta}\cap\dot H^{1-(\al-\beta)}}
   \ll \mu^{\kappa},}
if $\mu_0$ is sufficiently small depending on~$M$. Therefore,  we can identify $u$ as the fixed point for the iteration in the global Strichartz norm 
\EQ{\nn 
 \|u\|_{L^\I_t H^1_x(\R\times\R)} \lec M, \pq \|u\|_{L^p_tL^{2p}_x(\R\times\R)} \ll \mu^{\kappa},}
which automatically scatters. 
\end{proof}

Using the constants in Lemmas \ref{2nd exit}--\ref{0freq scat}, we choose $\e_*,\de_*,R_*,\mu>0$ such that 
\EQ{ \label{choice e R*}
 \pt \de_*\le\de_S,\pq \de_*\ll \de_X, \pq \e_*\le\e_0(\de_*), 
 \pr \e_* \ll R_*\ll\min(\de_*,\ka_1(\de_*)^{1/2},\ka_0^{1/2}\mu,J(Q)^{1/2}),}
\EQ{ \label{choice mu}
 \mu<\mu_0(M_*), \pq \mu^{\kappa}\ll J(Q)^{1/2}.} 

Suppose that a solution $u(t)$ on the maximal existence interval $I\subset\R$ satisfies for some $\e\in(0,\e_*]$, $R\in(2\e,R_*],$ and $\t_1<\t_2<\t_3\in I$, 
\EQ{\nn 
  E(\vec u)< J(Q)+\e^2, \pq d_Q(\vec u(\t_1))<R<d_Q(\vec u(\t_2))>R> d_Q(\vec u(\t_3)).}
Then there exist $T_1\in(\t_1,\t_2)$ and $T_2\in(\t_2,\t_3)$ such that 
\EQ{\nn 
 \pt d_Q(\vec u(T_1))=R=d_Q(\vec u(T_2))\le d_Q(\vec u(t))\pq (T_1<t<T_2).}
Lemma \ref{sign} gives us a fixed sign 
\EQ{
 \{\pm 1\}\ni \sg:=\Sg(u(t)) \pq (T_1<t<T_2).}
Based on the exact same virial identity as in \cite{NakS}, viz. 
\EQ{
 V_w(t):=\LR{wu_t|(x\na+\na x)u}, \pq \dot V_w(t)=-K_2(u(t)) +\text{error}
}
with some suitable time-dependent cut-off $w$ (for which the error is controlled by the energy
outside of some space-time rhombus)  we now obtain the following result. 

\begin{thm}[One-pass theorem] \label{no homo}
Let $\e_*,R_*>0$ be as in \eqref{choice e R*}. If a solution $u$ of NLKG on an interval $I$ satisfies for some $\e\in(0,\e_*]$, $R\in(2\e,R_*]$, and $\t_1<\t_2\in I$, 
\EQ{\nn
 E(\vec u)< J(Q)+\e^2, \pq d_Q(\vec u(\t_1))<R=d_Q(\vec u(\t_2)),}
then for all $t\in (\t_2,\I)\cap I=:I'$, we have $d_Q(\vec u(t))> R$. 
\end{thm}

Moreover, there exist disjoint subintervals $I_1,I_2,\dots\subset I'$ with the following property: On each $I_m$, there exists $t_m\in I_m$ such that 
\EQ{
 d_Q(\vec u(t))\simeq e^{k|t-t_m|}d_Q(\vec u(t_m)), \pq \min_{s=0,2}\sg K_s(u(t))\gec d_Q(\vec u(t))-C_*d_Q(\vec u(t_m)),}
where $\sg:=\Sg(\vec u(t))\in\{\pm 1\}$ is constant, $d_Q(\vec u(t))$ is increasing for $t>t_m$, decreasing for $t<t_m$, 
equals to $\de_X$ on $\p I_m$. For each $t\in I'\setminus\Cu_mI_m$ and $s=0,2$, one has  $(t-1,t+1)\subset I'$, $d_Q(\vec u(t))\ge\de_*$ and 
\EQ{ \label{int bd K}
 \int_{t-1}^{t+1}\min_{s=0,2}\sg K_s(u(t'))dt' \gg R_*^2.}

By the monotonicity, we can keep applying the above theorem at each $t>\t_2$ until $d_Q(\vec u)$ reaches $R_*$. 
Besides, one concludes that at  any later time $t_m>\t_2$ necessarily  $d_Q(\vec u)>R_*$. In other words, $u$ cannot return to the 
distance $R_*$ to $\pm Q$, after it is ejected to the distance $\de_X>R_*$.  

\subsection{Blowup in $\Sg=-1$}

If follows from the one-pass theorem  that the sign $\Sg$ stabilizes, i.e., if $u(t)$ exists on some maximal time-interval $I=[0,T_{*})$,
then there is $0<T_{0} \in I$ with the property that $\Sg(\vec u(t))=1$ or $-1$ on $T_{0}<t<T_{*}$. We now show by the usual Payne-Sattinger
convexity argument that necessarily $T_{*}<\I$ if $\Sg(\vec u(t))=-1$. 

Suppose not, and define $y(t)=\| u(t)\|_{2}^{2}$. Then
\[
\ddot y(t) = 2[\|\dot u\|_{2}^{2} - K_{0}(u(t))] 
\] 
implies together with the uniform negative upper bound on $K_{0}(u(t))$ for large times that
$y(t)\to\I$ as $t\to\I$. In particular, 
\[
\ddot y(t) = -2(p+1) E(\vec u) + 2p\|\dot u(t)\|_{2}^{2} + p\|u(t)\|_{H^{1}}^{2}\ge 2p \|\dot u(t)\|_{2}^{2}
\]
for large times, whence also, with $\al=\frac{p}{2}-1$,  
\[
\frac{d^{2}}{dt^{2}} y^{-\al} =-\al y^{-\al-2}(y\ddot y - (\al+1) \dot y^{2})<0
\]
for large times. But this is a contradiction to $y(t)\to\I$. 

\subsection{Global existence and scattering in $\Sg=1$}

It is clear that $T_{*}=+\I$ if $\Sg=1$. Indeed, this follows from the fact that $K_{0}(u(t))\ge0$ forces 
$$E(\vec u)\simeq \|u(t)\|_{H^{1}}^{2} + \|\dot u(t)\|_{2}^{2}$$
for all times. On the other hand, scattering is harder but can be obtained via the Kenig-Merle by essentially
the exact same arguments as in~\cite{NakS}. The only difference is that the Bahouri-Gerard~\cite{BaG} (as well as Merle-Vega~\cite{MeV})
decomposition and the perturbation lemma need to be stated in terms of Strichartz norms of the $1$-dimensional equation.
But this has already been done in~\cite{IMN}, and we see no reason to write this out again. Hence, we now obtain
the following scattering result in this case. 

\begin{thm}
For each $\e\in(0,\e_*]$, there exists $0<M(J(Q)+\e^2)<\I$ such that if a solution $u$ of NLKG on $[0,\I)$ satisfies 
$E(\vec u)\le J(Q)+\e^2$, $d_Q(\vec u(t))\ge R_*$ and $\Sg(\vec u(t))=+1$ for all $t\ge 0$, then $u$ 
scatters to $0$ as $t\to\I$ and $\|u\|_{L^p_tL^{2p}_x(0,\I)}\le M$. 
\end{thm}

The proof of Theorem~\ref{thm:9set} now follows   by the exact same arguments as in~\cite{NakS}. 

\section{Perturbation theory near the soliton for even solutions}

The goal of this section is to construct the center-stable manifold near $Q$ for 
even solutions of the NLKG equation~\eqref{eq:NLKG}. Due to the lack of adequate dispersion in one
dimension, we will need to invoke weighted estimates in our construction similar to those appearing in~\cite{Miz}. 
To formulate the theorem, we recall the basic hyperbolic/dispersive dynamics near the soliton.
Linearizing~\eqref{eq:NLKG} around $Q$ yields the operator
\EQ{\label{eq:L+}
 L_+ = -\p_{xx} +1 - p Q^{p-1}
}
which has a ground state $L_+\rho=-k^2\rho$, $k>0$, $\rho>0$, a zero eigenvalue $L_+\p_x Q=0$, no other eigenvalues, 
and no threshold resonance (provided $p>3$, which is the sharp condition for these properties), see for example~\cite[Lemma 9.1]{KS1}.  
Similarly,
\EQ{
\label{eq:L-}
L_{-}=-\p_{xx} +1 - Q^{p-1}
}
has no negative spectrum, $Q$ as a ground state since $L_{-}Q=0$, and no other eigenvalues, 
and no threshold resonance (again for $p>3$). 

We now seek solutions of the form\footnote{In contrast to the previous section we use $\mu$ here
since $\la$ will be used for the spectral parameter.} 
 $u=Q+v=Q+\mu\rho+w$ with $v$ small, and $w\perp \rho$. We write $\vec u=(u,\dot u)$.

\begin{prop}[Center-Stable manifolds]
\label{prop:CS}  Let $p\ge 5$. 
There exists $\nu>0$ small and a $C^1$ graph~$\M$ in $B_{\nu}(Q,0)\subset\HH$ so that $(Q,0)\in \M$, with tangent plane 
\EQ{\label{eq:TQM}
T_{Q}\M=\{(u_{0},u_{1})\in\HH\mid \langle k u_{0} + u_{1}|\rho\rangle =0\}
}
at $(Q,0)$ 
in the sense that 
\EQ{\label{eq:TM}
\sup_{x\in\partial B_{\delta}(Q,0)} \dist(x , T_{Q}\M)\lec \delta^{2} \quad \forall\; 0<\delta<\nu
}
Any data $(u_{0},u_{1})\in \M$ lead to global evolutions of~\eqref{eq:NLKG} of the form $u=Q+v=Q+\mu\rho+w$
where $v$ satisfies 
\EQ{\label{eq:vklein}
\|(v,\dot v)\|_{L^{\I}_{t} \HH} + \|v\|_{L^{p}((0,\I);L^{2p}(\R))}\lec\nu
}
and scatters to a free Klein-Gordon solution in~$\HH$, i.e., there exists a unique free Klein-Gordon solution $w_\I$ such that 
\EQ{\nn 
 |\mu(t)|+|\dot \mu(t)|+\|\vec w(t)-\vec w_\I(t)\|_\HH \to 0,} 
as $t\to\I$. In particular, we have $E(\vec u)=J(Q)+\|\vec w_\I\|_{\HH}^2/2$. 
Finally, any solution that remains inside $B_{\nu}(Q,0)$ for all $t\ge 0$ necessarily lies entirely on~$\M$, and $\M$ is invariant under the flow for all $t\ge0$.
\end{prop}

The regularity of the graph $\M$ is better than $C^1$ (depending on $p$), but for general $p$ the nonlinearity is not smooth and one can therefore
not expect $\M$ to be smooth (if the nonlinearity is given by an integer power such as $p=5,7$ then $\M$ is smooth). 
The remainder of this section is devoted to proving this result. As in~\cite{Miz}, \cite{Cuc}, we shall use
weighted $L^\I_x L^2_t$ bounds to overcome the weaker dispersion in one dimension. These estimates hinge on the fact
that $L_{\pm}$ have no threshold resonance which is reflected by the regularity of the spectral measure near zero energy. The latter
then guarantees faster {\em local} decay of the Klein-Gordon evolutions relative to $L_{\pm}$. 
To obtain these bounds we shall use the distorted Fourier transform. 

\subsection{The distorted Fourier transform}

We begin by recalling the distorted Fourier transform relative to
a Schr\"odinger operator on the line
\EQ{\label{eq:calLdef}
\calL:= -\frac{d^2}{dx^2} + V
}
with real-valued potential. 
In our case 
\EQ{\label{eq:Vcosh} 
V(x)=-\alpha \cosh^{-2}(\beta x)} for suitable $\alpha,\beta>0$, but for the 
moment we only need $V\in L^1_{\mathrm{loc}}(\R)$ and $\calL$ in the limit-point case at
 $\pm\I$. This material is of course standard, see for example Section~2 of~\cite{GZ}.  
Define $\phi_\al(x,x_0;z)$, $\theta_\al(x,x_0;z)$ to be the fundamental system of solutions of $$\calL \psi = z\psi,\quad  
z\in\C$$ so that
\EQ{\label{eq:phitheta}
\phi_\al(x_0,x_0;z)&=-\theta_\al'(x_0,x_0;z)= -\sin\al\\
\phi_\al'(x_0,x_0;z)&=\theta_\al(x_0,x_0;z)= \cos\al
}
where $x_0\in\R$ and $\al\in[0,\pi)$. Their Wronskian 
is
\[
 W(\theta_\al,\phi_\al)=1
\]
The {\em Weyl-Titchmarsh} solutions are defined as the unique solutions $\psi_{\pm,\al}(\cdot,x_0;z)\in L^2([x_0,\pm\I),dx)$
for $z\in\C\setminus\R$ 
which satisfy the boundary condition
\[
 \psi_{\pm}'(x_0,x_0;z)\sin\al + \psi_{\pm}(x_0,x_0;z)\cos\al  =1
\]
 This boundary condition ensures that 
\EQ{\label{eq:mdef}
\psi_{\pm,\al}(x,x_0;z) = \theta_\al(x,x_0;z) + m_{\pm,\al}(z,x_0)\phi_\al(x,x_0;z)
}
and the Wronskian
\[
 W(\psi_+(\cdot,x_0;z),\psi_-(\cdot,x_0;z)) = m_{-,\al}(z,x_0) - m_{+,\al}(z,x_0)
\]
The Weyl-Titchmarsh functions $m_{\pm,\al}$ are Herglotz functions, and the associated Weyl-Titchmarsh matrix
\EQ{\label{eq:Mal}
 M_\al(z,x_0):=\left[  \begin{matrix} 
\frac{1}{m_{-,\al}(z,x_0)-m_{+,\al}(z,x_0)} & \frac12 \frac{m_{-,\al}(z,x_0)+m_{+,\al}(z,x_0)}{m_{-,\al}(z,x_0)-m_{+,\al}(z,x_0)} \\
\frac12 \frac{m_{-,\al}(z,x_0) + m_{+,\al}(z,x_0)}{m_{-,\al}(z,x_0)-m_{+,\al}(z,x_0)} & \frac{m_{-,\al}(z,x_0)m_{+,\al}(z,x_0)}{m_{-,\al}(z,x_0)-m_{+,\al}(z,x_0)}
\end{matrix}
\right]
}
is a Herglotz matrix. This implies that there exists a nonnegative $2\times 2$-matrix valued measure $\Omega_\al(d\la,x_0)$ so that the representation 
\EQ{
 M_\al(z,x_0) &= C_\al(x_0) + \int_{\R} \big[ \frac{1}{\la-z} - \frac{\la}{1+\la^2} \big]\, \Omega_\al(d\la,x_0) \\
C_\al(x_0) &= C_\al(x_0)^*,\quad \int_{\R} \frac{\| \Omega_\al(d\la, x_0)\|}{1+\la^2} < \I
}
holds.  The measure $\Omega_\al$ satisfies
\EQ{\label{eq:Omrep}
\Omega_\al((\la_1,\la_2],x_0) = \pi^{-1} \lim_{\de\to0+}\lim_{\e\to0+} \int_{\la_1+\de}^{\la_2+\de} \Im M_\al(\la+i\e, x_0)\, d\la 
}
The matrix measure $\Omega_\al$ plays the role of the spectral measure, as can be seen from the following Fourier
representation relative to~$\calL$. 

\begin{prop}
\label{prop:FT} 
Let $\al\in[0,\pi)$, $f,g\in C_0^\I(\R)$, $F\in C(\R)\cap L^\I(\R)$, and $\la_1<\la_2$. Then
\EQ{
(f,F(H)E_H((\la_1,\la_2])g)_{L^2(\R;dx)} &= (\hat{f}_\al(\cdot,x_0), M_F M_{\chi_{(\la_1,\la_2]}}\hat{g}_\al(\cdot,x_0))_{L^2(\R,\Omega_\al(\cdot,x_0))} \\
&= \int_{(\la_1,\la_2]} \overline{\hat{f}_\al(\la,x_0)}^{T} \: \Omega_\al(d\la,x_0) \: \hat{g}_\al(\la,x_0)\, F(\la)
}
where
\EQ{
 \hat{h}_{\al,1}(\la,x_0) &= \int_{\R} h(x)\theta_\al(x,x_0;\la)\, dx,\quad \hat{h}_{\al,2} (\la,x_0) = \int_{\R} h(x)\phi_\al(x,x_0;\la)\, dx   \\
 \hat{h}_\al(\la,x_0) &= ( \hat{h}_{\al,1}, \hat{h}_{\al,2} )^{T}(\la,x_0)
}
This Fourier transform establishes a unitary correspondence between $L^2(\R,dx)$ and $L^2(\R,\Omega_0)$. 
\end{prop}

For example, for the free case one finds that, with $\al=0, x_0=0$, 
\EQ{\label{eq:free}
\phi_0(x,0;\la) &= \frac{\sin(\la^{\frac12} x)}{\la^{\frac12}} ,\quad \theta_0(x,0;\la) = \cos(\la^{\frac12} x) \\
m_{\pm,0}(z,0) &= \pm iz^{\frac12}, \quad z\in \C\setminus [0,\I) \\
\Omega_0(d\la,0) &= \frac{1}{2\pi}\chi_{(0,\I)}(\la) \left[ \begin{matrix} \la^{-\frac12} & 0 \\ 0 & \la^{\frac12} \end{matrix} \right]\, d\la
}
This can be seen to lead to the usual Fourier transform on the line, but written as $\sin,\cos$ transform, with
the Fourier variable in the positive half-axis. 

\subsection{The wave operators}

Throughout  this section, we will assume for simplicity that the potential $V$ in~\eqref{eq:calLdef} is a Schwartz function. 
Although much less is required for the following results to hold, it will be sufficient for our purposes to do so (since the
soliton is a Schwartz function). Recall that the wave operators
\[
 W_\pm =s-\lim_{t\to \pm\I} e^{it\calL} e^{it\p_x^2} 
\]
 exist and are isometries $L^2(\R)\to P_c(\calL)L^2(\R)$. Moreover, $W_\pm W^*_\pm =P_c(\calL)$, $W^*_\pm W_\pm = \mathrm{Id}$, and 
\EQ{\label{eq:intertwine}
 W^*_\pm  E_\calL W_\pm = E_0
}
as an identity between projection valued measures, where $E_0$, $E_\calL$ are the spectral resolutions of $-\p_x^2$ and $\calL$, 
respectively. One refers to~\eqref{eq:intertwine} as the ``intertwining property''. It is equivalent to the statement that
\EQ{\label{eq:inter2}
f(\calL)P_c(\calL) = W_\pm f(-\p_x^2) W_\pm^*
}
for every continuous, bounded $f$ on the line. 
Weder~\cite{Weder} and Artbazar, Yajima~\cite{AY}, proved that under our assumptions, the wave operators $W_{\pm}$ are 
bounded on all $W^{k,p}$, $k\ge0$, $1<p<\I$. This is very useful, as it allows one to transfer estimates from the free case to~$\calL$
by means of~\eqref{eq:inter2}.  For example, one has the following multiplier result.

\begin{lem}
 \label{lem:Mikhlin} 
Let $m:(0,\infty)\to\C$ satisfy 
\EQ{\label{eq:Mcond}
\sup_{\la>0} [|m(\la)|+\la|m'(\la)| ] \le C
}
Then $m$ is a bounded Fourier multiplier on $P_c(\calL) L^p(\R)$ for $1<p<\I$.  By ``Fourier multiplier'' we mean a function $m(\la)$ that
multiplies both components of the vector $\hat{h}_\al$ in Proposition~\ref{prop:FT}. 
\end{lem}
\begin{proof}
With $\calF$ denoting the distorted, and $\calF_0$ the free (as in~\eqref{eq:free}) Fourier transforms, respectively, one has
\[
 \calF^{-1}\circ \calM_m \circ \calF P_c (\calL)  = W_\pm \circ \calF_0^{-1} \circ \calM_m \circ \calF_0 \circ W_\pm^*
\]
where $\calM_m$ is the component-wise multiplication operator by $m(\la)$. 
\end{proof}

\subsection{Specializing to $L_\pm$}

Now we consider the case of $\calL:=L_\pm -1=-\frac{d^2}{dx^2}+V$ where $V$ is as in~\eqref{eq:Vcosh}. In fact,
we shall only use here that $V$ is an even Schwartz function and that $\calL$ has no zero energy resonance.  Recall that this means
that there is no globally bounded solution of $\calL f=0$.  Note that the free case does have a zero energy resonance, as
does 
\[
 \calL = -\frac{d^2}{dx^2} - 6\cosh^{-2}(x)
\]
which is the operator one obtains from linearizing the cubic NLKG equation.  

\begin{lem}
 \label{lem:L_+} 
Suppose that $\calL$ as in Proposition~\ref{prop:FT}  has an even Schwartz potential and no zero energy resonance. 
Then the spectral measure $\Omega_0$ is diagonal, absolutely continuous on $(0,\I)$,  and of the form 
\EQ{
\Omega_{0}(d\la,0) &= \diag( O(\la^{\frac12}), O(\la^{\frac12}))\,d\la, \quad \la\to0 \\
\Omega_0(d\la,0) &= \diag( O( \la^{-\frac12}), O(\la^{\frac12} )) \, d\la, \quad \la\to\I
}
on $\la>0$. The $O(\cdot)$-terms satisfy 
  the natural derivative bounds. 
\end{lem}
\begin{proof}
Taking $\al=0$ and $x_0=0$ in~\eqref{eq:phitheta} (and suppressing $x_0$)  yields 
\[
 \theta_0(x;z)=\theta_0(-x;z),\quad \phi_0(x;z)=-\phi_0(-x;z),\quad \psi_{-,0}(x;z)= \psi_{+,0}(-x;z)
\]
whence $m_{-,0}(z)=-m_{+,0}(z)$ and $W(z)=2m_{-,0}(z)$. Denote by $u_{0,+}(x)$ and $u_{1,+}(x)$ a fundamental system of solutions
to $\calL f=0$ with 
\EQ{\label{eq:u01+}
u_{0,+}(x) &= 1 + O(x^{-100}) \\
u_{1,+}(x) &= x + O(x^{-100})
}
as $x\to\I$.  This representation follows from the Volterra integral equations
\EQ{\nn
u_{0,+}(x) &= 1 + \int_{x}^{\I} (y-x) V(y) u_{0,+}(y)\, dy \\
u_{1,+}(x) &= x + \int_{x}^{\I} (y-x) V(y) u_{1,+}(y)\, dy
}
by iteration. 
In particular, $W(u_{1,+},u_{0,+})=1$.
Furthermore, $u_{0,-}, u_{1,-}$ denote the corresponding solutions, but with $x\to-\I$. By symmetry, $u_{0,-}(x)=u_{0,+}(-x)$ and $u_{1,-}(x)=-u_{1,+}(-x)$.
Since zero energy is nonresonant, $W(u_{0,+}, u_{0,-})\ne0$.  Perturbatively in $\la$, we now obtain from $u_{j,+}(x)$ unique 
eigenfunctions $u_{j,+}(x,\la)$ satisfying $\calL u_{j,+} =\la u_{j,+}$, as well as 
 for small $\lambda$ and $|\la x^{2}|\ll 1$, 
\EQ{\nn 
u_{j,+}(x,\la) &= u_{j,+}(x)(1+O(\la x^{2})) \qquad j=0,1
}
Indeed, $u_{j,+}(x,\la)$ are given in terms of the Volterra equations
\EQ{\nn 
u_{j,+}(x,\la)  = u_{j,+}(x) + \la \int_{0}^{x} [u_{0,+}(x)u_{1,+} (y) - u_{0,+}(y)u_{1,+} (x)]\, u_{j,+}(y,\la) \, dy
}
Similarly,  the Jost solutions $f_{+}(x,\la)$ defined by $\calL f_{\pm}(\cdot,\la)=\la  f_{\pm}(\cdot,\la)$, $f_{\pm}(x,\la)\simeq e^{\pm ix\la^{\frac12}}$ as $x\to\pm\I$, satisfy
\[
f_{\pm}(x,\la) = e^{\pm ix\la^{\frac12}}(1+O(x^{-100})) \qquad \pm x\gg 1
\]
As usual, this follows from the Volterra representation of these functions. 
Moreover, one has 
\[
f_{\pm}(x,\la) = a_{\pm}(\la) u_{0,\pm }(x,\la) + b_{\pm} (\la) u_{1,\pm}(x,\la)
\]
with 
\[
a_{\pm}(\la) = -W(f_{\pm}(\cdot,\la), u_{1,\pm}(\cdot,\la)),\quad  b_{\pm}(\la) = W(f_{\pm}(\cdot,\la), u_{0,\pm}(\cdot,\la))
\]
Therefore,  for any small $\e>0$ 
\EQ{
a_{\pm}(\la) &=  1 +  O(\la^{1-\eps}) \\
b_{\pm}(\la) &= \mp i\la^{\frac12} +O(\la^{1-\eps})
}
as $\la\to0$. In conclusion, 
\EQ{\label{eq:Wla}
W(\la):=W(f_{+}(\cdot,\la),f_{-}(\cdot,\la)) = c_{0}+ ic_{1}\la^{\frac12} + O(\la^{1-\eps}) 
}
where $c_{0},c_{1}\in\R$ with $c_{0}\ne0$. The matrix in~\eqref{eq:Mal} is
\[
M_{0}(\la) = \diag( W(\la)^{-1},\frac14 W(\la)) 
\]
and the measure $\Omega_{0}(\la)$ satisfies, for small $\la$, by~\eqref{eq:Omrep}
\EQ{\label{eq:Omegala}
\Omega_{0}(d\la) = \diag( O(\la^{\frac12}), O(\la^{\frac12}))\,d\la, \quad \la\to0
}
For large $\la$, the free representation \eqref{eq:free} describes $\Omega_{0}$ to leading order. 
\end{proof}

\subsection{Evolution estimates} 

We now establish some estimates on the   Klein-Gordon evolution relative to $L_{\pm}$ on the line. 
These are analogous to those for the Schr\"odinger evolution obtained in~\cite{Miz}.  

\begin{lem}\label{lem:evol}
Let $L_{\pm}$ be as in \eqref{eq:L+}, \eqref{eq:L-}, respectively, with $p>3$. Then 
one has the following bounds 
\EQ{\label{eq:KG1}
\Big\| \lan x\ran^{-1} e^{\pm it  L_{\pm}^{\frac12} } P_{c}(L_{\pm}) f \Big \|_{L^{\I}_{x} L^{2}_{t}} &\lec \| \lan \p_{x}\ran^{\frac12}  f\|_{2} \\
\Big\| \lan x\ran^{-1} \int_{-\I}^{t} e^{\pm i(t-s)  L_{\pm}^{\frac12} } P_{c}(L_{\pm}) f(s)\, ds \Big \|_{L^{\I}_{x} L^{2}_{t}} &\lec  \int_{-\I}^{\I} \| \lan \p_{x}\ran^{\frac12}  f(s)\|_{2}\, ds
}
for all Schwartz functions $f$ on the line. 
\end{lem}
\begin{proof} We begin with the first. Write $w(x)=\lan x\ran^{-1} $. Then by duality we need to estimate, with $g=g(t,x)$, and $\calF$
denoting the distorted Fourier transform of Proposition~\ref{prop:FT}, 
\EQ{\label{eq:fgdual}
&\Big\lan e^{\pm it  L_{\pm}^{\frac12} } P_{c}(L_{+}) f | w g\Big\ran = \int_{-\I}^{\I} dt \int_{0}^{\infty}   
 e^{\pm it(1+\la)^{\frac12}} \calF {f}(\la)^{T} \Omega_{0} (d\la) \calF( w  g(t))(\la)  \\
 &= \int_{-\I}^{\I} dt \int_{0}^{\infty}   
 e^{\pm it (1 + \la )^{\frac12}} \calF_{1} {f}(\la) \mu_{1} (\la) \int_{-\I}^{\I}  w(x)  g(t,x) \theta(x,\la) \, dx\,d\la\\ 
 &+  \int_{-\I}^{\I} dt \int_{0}^{\infty}    
 e^{\pm it (1+  \la )^{\frac12}} \calF_{2} {f}(\la) \mu_{2} (\la) \int_{-\I}^{\I}  w(x)  g(t,x) \phi(x,\la)\, dx \, d\la
}
where $\Omega_{0}(d\la)=\diag(\mu_{1},\mu_{2})\,d\la$, and $\theta,\phi$ are $\theta_{0}(x,0;z)$, $\phi_{0}(x,0;z)$ from above. 
Here $\la$ denotes the spectral variable of $\calL=L_{\pm }-1$. 
For small $\la$ one has  $$(1+\la)^{\frac12}=1+\frac12\la + O(\la^{2})$$ Up to change of variable in $\la$ (which we ignore) we can
therefore view the small~$\la$ integral as the usual Fourier transform. Denoting the Fourier transform of $g(t,x)$ in time by $\hat{g}(\la,x)$
we can bound the contribution of small $\la$ to~\eqref{eq:fgdual}  by means of~\eqref{eq:Omegala} as follows:
\EQ{
&\int_{0}^{1}  \int_{-\I}^{\I}  | \calF_{j} {f}(\la)| \,  |\hat{g}(\la,x)| [|\theta(x,\la)|+|\phi(x,\la)|] \, w(x)dx\, \mu_{j}(\la)\,  d\la\\
&\lec \sup_{0<\la<1}\sup_{x}w(x) \mu_{j}(\la)^{\frac12} [|\theta(x,\la)|+|\phi(x,\la)|]  \int_{-\I}^{\I} \Big( \int_{0}^{1}  | \calF_{j} {f}(\la)|^{2}   \mu_{j}(\la)\,  d\la\Big)^{\frac12}\\
&\qquad\qquad \times
\Big( \int_{0}^{1}  | \hat{g}(\la,y) |^{2}    \,  d\la\Big)^{\frac12}\, dy \lec \|f\|_{2} \|g\|_{L^{1}_{x}L^{2}_{t}}
}
In the final step we used the unitarity of the Fourier transform, both in the free and distorted cases, 
as well as the fact that uniformly in $\la>0$
\[
 \sup_{x} \lan x\ran^{-1} [|\theta(x,\la)|+|\phi(x,\la)|] \lec 1
\]
For $\la>1$, one has $(\la \ran^{\frac12}\simeq \la^{\frac12}$.
Hence, the first component for these $\la$-values is bounded by (we can ignore $w(x)$ in this regime, as well as $\theta$, cf.~\eqref{eq:free}), 
\EQ{
&\int_{1}^{\I} \int_{-\I}^{\I}  |\calF_{1} f(\la)| |\hat{g}(\sqrt{\la},x)| \mu_{1}(\la) \, dx d\la \\
& \lec \int_{1}^{\I} \int_{-\I}^{\I}  |\calF_{1} f(\la^{2})| |\hat{g}(\la,x)| \,\la \mu_{1}(\la^{2}) \, dx d\la \\
&\lec \int_{-\I}^{\I}   \Big(\int_{1}^{\I}  |\calF_{1} f(\la)|^{2}  \la \mu_{1}^{2}(\la)\,    d\la\Big)^{\frac12}   \Big(\int_{1}^{\I} |\hat{g}(\la,x)|^{2}\,  d\la\Big)^{\frac12} \, dx\\
&\lec \| \lan \p_{x}\ran^{\frac12} f\|_{2} \|g\|_{L^{1}_{x}L^{2}_{t}}
}
The final estimate here follows by the free asymptotics~\eqref{eq:free}. 

For the second inequality in the lemma one proceeds in a similar fashion. In fact, using the notation from the first part of the proof, 
\EQ{\nn
&\Big\lan  \int_{-\I}^{t} e^{\pm i(t-s)  L_{\pm}^{\frac12} } P_{c}(L_{\pm}) f(s)\, ds | w g\Big\ran \\
& = \int_{-\I}^{\I} dt \int_{-\I}^{t} ds \int_{-\I}^{\infty}   
 e^{\pm i(t-s)(1+ \la)^{\frac12}} \calF ({f}(s))(\la)^{T} \Omega_{0} (d\la) \calF (w  g(t))(\la)  \\
 &= \int_{-\I}^{\I} ds \int_{s}^{\I} dt \int_{0}^{\infty}   
 e^{\pm i(t-s) (1+\la)^{\frac12}} \calF_{1} ({f}(s)) (\la) \mu_{1} (\la) \int_{-\I}^{\I}  w(x)  g(t,x) \theta(x,\la) \, dx\,d\la\\ 
 &+  \int_{-\I}^{\I} ds \int_{s}^{\I} dt \int_{0}^{\infty}    
 e^{\pm i(t-s)(1+\la)^{\frac12}} \calF_{2} ({f}(s)) (\la) \mu_{2} (\la) \int_{-\I}^{\I}  w(x)  g(t,x) \phi(x,\la)\, dx \, d\la
}
Carrying out the $t$-integration, and performing similar arguments as in the previous case, shows that
the two final expressions here are
\[
 \lec \int \| \lan \p_{x}\ran^{\frac12} f(s)\|_{2}\, ds\;  \|g\|_{L^{1}_{x}L^{2}_{t}}
\]
as desired. 
\end{proof}

Note that \eqref{eq:KG1} cannot hold for the free case, i.e., $e^{it\lan \p_{x}^{2}\ran^{\frac12}}$ since the best point-wise decay
of the latter evolution is $t^{-\frac12}$. For the nonlinear analysis, it will be technically advantageous to work in $L^{2}_{x}$ rather
than with $L^{\I}_{x}$ on the left-hand side of~\eqref{eq:KG1}. This can easily be done just by switching to faster decaying weights.
 In addition, we shall add half of a derivative to the left-hand side, reflecting
the fact that we are working in the energy space $H^{1}$.  As standard derivatives do not commute with $L_{+}$, this step requires
some care. While one could use powers of $L_{+}$, via Lemma~\ref{lem:Mikhlin}, we instead rely on some algebra involving $L_{-}$. 
The precise statement is as follows. Recall that $\r$ is the ground state of $L_{+}$,  and that $Q'$ satisfies $L_{+} Q'=0$. 

\begin{cor}
\label{cor:upgrade} 
Any solution of 
\EQ{ \label{eq:L+u}
 \ddot u=-L_+u + F,\;\ u\perp \r, Q' 
}
satisfies 
\EQ{\label{eq:uF}
 \|\LR{x}^{-s}u\|_{L^2_tH^{1/2}_x} \lec \|u(0)\|_{H^1}+\|\dot u(0)\|_{L^2} + \| F\|_{L^{1}_{t}L^{2}_{x}}
 }
 with $s>\frac32$. 
\end{cor}
\begin{proof}
By Lemma~\ref{lem:evol}, any $u$ as in \eqref{eq:L+u} satisfies 
\EQ{ \label{est u}
 \|\LR{x}^{-1}u\|_{L^\I_xL^2_t}\lec\|u(0)\|_{H^{1/2}}+\|\dot u(0)\|_{H^{-1/2}} + \| F\|_{L^{1}_{t}H^{-1/2}} }
The left-hand side dominates 
\EQ{ 
 \|\LR{x}^{-s}u\|_{L^2_tL^2_x} \pq (s>3/2).}
Similarly,  any solution of $\ddot v=-L_-v + \tilde F$ with $\ v\perp Q$ satisfies 
\EQ{ \label{est v}
  \|\LR{x}^{-s}v\|_{L^2_{t,x}}\lec \|v(0)\|_{H^{1/2}}+\|\dot v(0)\|_{H^{-1/2}} + \| \tilde F\|_{L^{1}_{t}H^{-1/2}} .}
In order to estimate the derivative, we use the special one-dimensional property
\EQ{
 U:=\r\p_x\r^{-1}, \pq \r=CQ^{(p+1)/2}, \pq UL_+=L_-U, \pq U^*Q=-CQ_x, }
where $C$ denotes positive constants dependent on $p$. 

In particular, with $u$ as before,  $v:=Uu$ and $\tilde F= UF$ satisfies \eqref{est v}, whence
\EQ{
 \|\LR{x}^{-s}u\|_{L^2_tH^1_x}
 \pt\lec \|\LR{x}^{-s}u\|_{L^2_{t,x}}+\|\LR{x}^{-s}Uu\|_{L^2_{t,x}} 
 \pr\lec\|u(0)\|_{H^{1/2}}+\|\dot u(0)\|_{H^{-1/2}}+\|Uu(0)\|_{H^{1/2}}+\|U\dot u(0)\|_{H^{-1/2}} 
 \pr \qquad + \| F\|_{L^{1}_{t}H^{-1/2}} + \| \tilde F\|_{L^{1}_{t}H^{-1/2}}
 \pr\lec \|u(0)\|_{H^{3/2}} + \|\dot u(0)\|_{H^{1/2}} +  \| F\|_{L^{1}_{t}H^{1/2}} }
Interpolating it with the estimate without the derivative yields 
\EQ{
 \|\LR{x}^{-s}u\|_{L^2_tH^{1/2}_x} \lec \|u(0)\|_{H^1}+\|\dot u(0)\|_{L^2} + \| F\|_{L^{1}_{t}L^{2}_{x}} 
 }
as claimed.
\end{proof}

The  significance of \eqref{eq:uF} lies with the perturbative nonlinear analysis of Section~\ref{sec:nonlinear}. 
In fact, placing a nonlinear term such as $Q^{q}u^{2}$ in $L^{1}_{t}L^{2}_{x}$ yields 
\EQ{\label{eq:Qnice}
 \|Q^q u^2\|_{L^1_tL^2_x} \lec \|\LR{x}^{-s}u\|_{L^2_tL^4_x}^2 \lec \|\LR{x}^{-s}u\|_{L^2_t H^{\frac12}_x}^2 }
for $H^1$ solutions $u$. This will allows us to close our estimates very easily.

We shall also require Strichartz estimates on the Klein-Gordon equation relative to $\calL$.
Note that unlike Lemma~\ref{lem:evol}, Corollary~\ref{cor:calLStr} has nothing to do with $\calL$ having a zero
energy resonance or not. 

\begin{cor}
\label{cor:calLStr} 
For any Schwartz function $u$ in $\R^{1+1}_{t,x}$ with $u=P_c(\calL)u$ one has the aprori bound 
\EQ{
\label{eq:calLstr}
\| u\|_{L^p_t L^r_x} \lec   \| u[0]\|_{H^1\times L^2} + \| \ddot u + L_{\pm} u  \|_{L^1_t L^2_x} 
}
for any $4<p\le \I$, $0<\frac{1}{r}\le \frac12-\frac{2}{p}$. In particular, one can take $r=2p$ for any $5\le p<\I$. 
\end{cor}
\begin{proof}
We first recall the following Strichartz estimates for the free case, see for example~\cite[Section~4.2]{IMN}: for any Schwartz function $u$ in $\R^{1+1}_{t,x}$ there is the aprori bound 
\EQ{
\label{eq:STR1}
& \| u\|_{L^4_t B^{\frac14}_{\I,2}} + \| u\|_{L^\I_t H^1} + \|\dot u\|_{L^\I L^2}  \lec \| u[0]\|_{H^1\times L^2} + \| \Box u + u \|_{L^{\frac43}_t B^{\frac34}_{1,2}+L^1_t L^2_x}
}
Here $B^\sigma_{p,2}$ are the usual inhomogeneous Besov spaces. By interpolation, one obtains the following bounds: for any $4\le p\le\I$, and any $2\le q\le \I$,  
\EQ{
\label{eq:STR2}
\| u\|_{L^p_t B^\alpha_{q,2}} \lec   \| u[0]\|_{H^1\times L^2} + \| \Box u + u \|_{L^{\frac43}_t B^{\frac34}_{1,2}+L^1_t L^2_x} 
}
where $\al=1-\frac{3}{p}$, $\frac{1}{q}=\frac12-\frac{2}{p}$.  By the embedding $B^\al_{q,2}\hookrightarrow L^r(\R)$ for any $q\le r<\I$ provided $\al\ge \frac{1}{q}-\frac{1}{r}$,
we now conclude that 
\EQ{
\label{eq:STR3}
\| u\|_{L^p_t L^r_x} \lec   \| u[0]\|_{H^1\times L^2} + \| \Box u + u \|_{L^{\frac43}_t B^{\frac34}_{1,2}+L^1_t L^2_x} 
}
for any $4<p\le \I$, $0<\frac{1}{r}\le \frac12-\frac{2}{p}$. 
 Finally, \eqref{eq:calLstr} follows from \eqref{eq:STR3} and the $L^r$-boundedness of the wave operators of $\calL$ for $1<r<\I$. 
\end{proof}

The condition $p\ge5$ is significant for the Strichartz-based small
data scattering theory: indeed, consider the equation 
\[
 \Box u+u=\pm |u|^{p-1} u
\]
on the line. Then placing the nonlinearity in $L^1_t L^2_x$ requires Strichartz control on $\| u\|_{L^p_t L^{2p}_x}$. However, the latter means
that we need $\frac{1}{2p}\le \frac12-\frac{2}{p}$, or $p\ge5$ (the latter being the $L^2$-critical power in dimension~$1$).

\subsection{Proof of Proposition~\ref{prop:CS}} \label{sec:nonlinear}
We shall now show  that for even functions one can construct the center-stable manifold for
\EQ{\label{eq:NLKG7}
\Box u + u = |u|^{p-1}u
}
in $1$-dim where $p> 5$ is fixed. 
Writing $u=Q+v$ with $v$ small, one has
\EQ{\label{eq:linearizedNLKG}
\ddot v + L_{+} v = N(Q,v)
}
where $L_{+}=-\p_{x}^{2}+1 - p Q^{p-1}$. The nonlinearity $N$ satisfies the bound
\EQ{\label{eq:Nest}
|N(Q,v)|\lec Q^{p-2} |v|^{2} + |v|^{p} 
}
The operator $L_{+}$ has negative spectrum, $L_{+}\r = -k^{2}\r$ and one therefore has to write $v(t,x)=\mu(t) \rho + w(t,x)$. The resulting
equations are
\EQ{\label{eq:muwsys}
(\ddot\mu - k^2 \mu)(t)\rho &= P_\rho N(Q,v) \\
\ddot w + P_\rho^\perp L_{+} w &= P_\rho^\perp N(Q,v)
}
To control the $w$-equation we shall use the Strichartz norm
\EQ{\label{eq:Snorm}
\|w\|_{S}:= \| (w,\dot w)\|_{H^{1}\times L^{2}}+ \|w\|_{L^{p}_{t} L^{2p}_{x}}
} 
as well as the localized norm
\EQ{\label{eq:Lnorm}
\| w\|_{L}  :=    \| \lan x\ran^{-2}  w\|_{L^{2}_{t,x}}
}
We now close  the estimates for $w$ in the space $\| w\|_{X}=\|w\|_{S}+ \|w\|_{L}$, and those for $\mu$ in the space $L^1\cap L^\infty([0,\I))$. 
Applying the Strichartz estimates for $L_{+}$ and using~\eqref{eq:Qnice} yields
\EQ{\label{eq:P1}
\| w\|_{S} &\lec \| w[0]\|_{\HH} + \| w\|_{L}^{2} + \| w\|_{S}^{p}  
\lec \delta + \|w\|_{X}^{2}
}
where $\|w[0]\|_{\HH}<\delta<\nu$. 
Note that we are assuming the desired bound on $v$ appearing on the right-hand side of the equation, in keeping with the usual contraction argument setup. 
To bound the local norm, we use Corollary~\ref{cor:upgrade}  to conclude that 
\EQ{\label{eq:P2}
 \| \lan x\ran^{-2}  w \|_{L^{2}_{t,x}}  &\lec \| \tilde v[0]\|_{\calH} 
 +  \| N(Q,v)\|_{L^{1}_{t}L^{2}_{x}}
 }
which can be estimated as before. Combining~\eqref{eq:P1} with~\eqref{eq:P2} yields 
\EQ{
\label{eq:wbd}
\|w\|_{X} \lec \delta + \|v\|_{X}^{2}
}
For the $\mu$ equation in \eqref{eq:muwsys} we introduce
\[
 \mu_\pm = \frac12(\mu\pm \frac{1}{k}\dot\mu),\qquad \mu=\frac12(\mu_+ +\mu_-)
\]
which implies that
\EQ{   \label{eq:muODE}
\dot \mu_\pm &= \frac12(\dot\mu\pm k\mu) \pm \frac{1}{2k} P_\rho N(Q,v)  \\
&= \pm k\mu_\pm \pm \frac{1}{2k} P_\rho N(Q,v) 
}
The homogeneous solutions are $\mu_\pm(t)= e^{\pm kt}$. Under the stability condition
\EQ{\label{eq:stabcond}
0= \mu_+(0) +  \frac{1}{2k} \int_0^\I e^{-ks}  P_\rho N(Q,v)(s)\, ds
}
the solutions to \eqref{eq:muODE} are
\EQ{
\label{eq:musol}
\mu_+(t) &= -\frac{1}{2k}\int_t^\infty e^{-k(s-t)}  P_\rho N(Q,v)(s)\, ds \\
\mu_-(t) &= e^{-kt} \mu_-(0) + \frac{1}{2k}\int_0^t e^{-k(t-s)} P_\rho N(Q,v)(s)\, ds
}
Therefore, assuming~\eqref{eq:stabcond}, one obtains the estimates
\EQ{ \nn 
\| \mu\|_{L^1\cap L^\I((0,\I))} &\lec |\mu_-(0)|+ \int_0^\I | \lan N(Q,v)(s),\rho\ran|\, ds \\
&\lec \de+ \| N(Q,v)\|_{L^{1}_{t } L^{2}_{x}} \lec \de + \|v\|_X^2
}
It is now straightforward to show that the iteration (or, the fixed-point argument) yields a unique solution $(\mu, w)\in X_{0}\times X$,
where $X_{0}= L^{1}\cap L^{\I}((0,\I))$ and $X$ is as before, with $\| (\mu,w)\|_{X_{0}\times X}\lec \de$. Moreover, $w$ scatters with
scattering data
\[
w_{\I}:= w(0) + \int_{0}^{\I} \frac{\sin(t\omega)}{\omega} N_{\r}^{\perp}(Q,v)(t)\, dt
\]
where $\omega:=(P_{c}(L_{+}) L_{+})^{\frac12}$. 

For the uniqueness, we note that if $(\mu,w)$ are such that 
\[
\sup_{t\ge0} [|\mu(t)| + |\dot\mu(t)| + \|w(t)\|_{H^{1}} + \| \dot w(t)\|_{L^{2}} ] \lec \de
\]
then necessarily \eqref{eq:stabcond} holds, whence also \eqref{eq:musol}. 
Denoting for any finite $T>0$
\[
\| (\mu,w)\|_{X_{T}} := \| \mu \|_{L^{1}(0,T)} + \|\mu \|_{L^{\I}(T,\I)} + \|w\|_{S(0,T)}
\]
we thus conclude from the integral equations that 
\EQ{
\sup_{T>0} \| (\mu ,w)\|_{X_{T}} \lec \delta
}
Hence also $\| (\mu,w)\|_{X}\lec \delta$ and one obtains the claim from the uniqueness of the fixed point. 
Finally, \eqref{eq:TM} holds since the distance on the left-hand side of that inequality is proportional to $\mu_{+}(0)$
as given by~\eqref{eq:stabcond}. Since the integral in the expression is of size $O(\de^{2})$, we are done.

\subsection{Proof of Theorem~\ref{thm:main} and Corollary~\ref{cor:DM}}

We only need to address the behavior of those solutions in Theorem~\ref{thm:9set} which are trapped as $t\to\I$. 
If so, then the uniqueness part of Proposition~\ref{prop:CS} states that $\vec u(t)$ lies on the manifold $\calM$ constructed 
in that proposition for all large~$t$. Therefore, the set of all data which lead to solutions trapped by $Q$ in forward times is
the maximal backward evolution of $\calM$ under the flow of~\eqref{eq:NLKG}. This can easily be seen to be a $C^1$  manifold (in fact, it has better smoothness). 
The scattering to~$Q$ was shown in Proposition~\ref{prop:CS}, which concludes the proof of Theorem~\ref{thm:main}.

For the corollary, we proceed as in \cite{NakS}. Thus, if $E(\vec u)=E(Q,0)$ and if the solution is trapped, then $\vec w_{\I}=0$. Therefore, the entire dynamics is
controlled by a single quantity, namely $\mu_{-}(t)$. This in turn is determined uniquely by $\mu_{-}(0)$. It is clear that $\mu_{-}(0)=0$ forces
$\mu_{-}\equiv0$, or in other words, $u=Q$ for all times. If $\mu_{-}(0)\ne0$, then it can never change sign which leads to the two solutions $W_{\pm}$. Since
these are clearly one-dimensional and time-translation is a symmetry for these special solutions, they are all obtained by time-translation of two fixed representatives
of solutions converging to $Q$. Since the sign of $K_{0}$ is the same as that of $-\mu_{-}$ upon ejection from a neighborhood of $(Q,0)$, we see that 
$W_{\pm}$ have the behavior as $t\to-\I$ as described in Corollary~\ref{cor:DM}.

\end{document}